\documentclass[a4paper,12pt]{amsart}
\usepackage{amsmath,amssymb}
\usepackage[left=2cm,right=2cm]{geometry}
\def\C{\Bbb C}
\def\CC{\mathcal C}
\def\hol{\mathcal O}

\def\O{\Omega}

\def\spc{$\indent$}

\def\k{\kappa}
\def\endofproof{\hfill \square}

\def\vp{\varphi}
\def\ov{\overline}

\def\lk{ l^{\kappa}}
\def\Cn{\mathbb{C}^n}

\newtheorem{thm}{Theorem}
\newtheorem{prop}[thm]{Proposition}
\newtheorem{cor}[thm]{Corollary}
\newtheorem{defn}[thm]{Definition}
\newtheorem{lem}[thm]{Lemma}
\title[Local and global notions of visibility]
{Local and global notions of visibility with respect to Kobayashi distance, a comparison}

\author{Nikolai Nikolov, Ahmed Yekta Ökten, Pascal J. Thomas} 
\address{N. Nikolov\\
	Institute of Mathematics and Informatics\\
	Bulgarian Academy of Sciences\\
	Acad. G. Bonchev Str., Block 8\\
	1113 Sofia, Bulgaria}

\address{Faculty of Information Sciences\\
	State University of Library Studies
	and Information Technologies\\
	69A, Shipchenski prohod Str.\\
	1574 Sofia, Bulgaria}

\email{nik@math.bas.bg}

\address{A. Y. Ökten\\
	Institut de Math\'ematiques de Toulouse; UMR5219 \\
	Universit\'e de Toulouse; CNRS \\
	UPS, F-31062 Toulouse Cedex 9, France} \email{ahmed$\_$yekta.okten@math.univ-toulouse.fr}

\address{P. J. Thomas\\	Institut de Math\'ematiques de Toulouse; UMR5219 \\Universit\'e de Toulouse; CNRS \\	UPS, F-31062 Toulouse Cedex 9, France} \email{pascal.thomas@math.univ-toulouse.fr}

\thanks{The authors were partially supported by the Bulgarian National Science Fund,
	Ministry of Education and Science of Bulgaria under contract KP-06-Rila/2 and by Campus France under the contract PHC Rila 48135TJ, Invariant Functions In Complex Analysis. \\
	The second author received support from the University Research School EUR-MINT
	(State support managed by the National Research Agency for Future Investments
	program bearing the reference ANR-18-EURE-0023). \\
	This paper was written during the visit of second and third author at Institute of Mathematics and Informatics, Bulgarian Academy of Sciences.}


\begin{document}


\begin{abstract} In this note, we introduce the notion of visible boundary points with respect to Kobayashi distance for domains in $\Cn$. Following the work of Sarkar \cite{S}, we obtain additive and multiplicative localization results about Kobayashi distance near visible boundary points. Then using the additive localization result, we show that visibility property with respect to Kobayashi distance is a local property of the boundary points and it doesn't depend on the domain.
\end{abstract}

\maketitle
\section{Introduction}
The definition of invariant distances in complex analysis stems from the properties of holomorphic mappings,
which are defined on open sets. In one or several dimensions, the question of extension of those mappings to 
the boundary of the open set (and to which boundary?) is of interest.  
The work of Balogh and Bonk \cite{BB} introduced the metric property of Gromov 
hyperbolicity into this subject, along with 
an identification of the Gromov boundary with the Euclidean boundary, to provide another proof of the Fefferman
extension theorem for mappings of strictly pseudoconvex domains. \\
\spc Another geometric property that geodesic spaces may have, which can serve
as a sort of substitute to Gromov hyperbolicity, is  visibility. Informally speaking, a metric space 
satisfies visibility property if geodesics joining  points approaching distinct points on the boundary pass through a compact set depending on those distinct boundary points. 
However, in general, it is not known whether geodesics for the Kobayashi-Royden pseudometric exist. Bharali and Zimmer \cite{BZ} introduced a wider notion of visibility which holds for almost-geodesics with respect to 
the Kobayashi-Royden metric. They also established some sufficient conditions for it in terms of the growth of the Kobayashi distance and Kobayashi-Royden pseudometric.\\ 
\spc This was followed by 
extensions of those results  in \cite{BM}, and other examples of sufficient and necessary 
conditions for visibility in \cite{BNT} and \cite{CMS}. Most of the results in those papers relied on conditions
which were local in terms of the Euclidean boundary of the domain, that is to say that e.g. to prove
visibility of a domain $\Omega$ (with respect to $\partial \Omega$) one requires properties to be tested on $\Omega \cap U_p$, for all 
$p\in \partial \Omega$, where $U_p$ is an appropriate neighborhood of $p$ in $\C^n$.  As the 
Kobayashi distance depends on the domain, \cite{BNT} needed some localization results about Kobayashi distances
along the lines of \cite{FR}. \\ 
\spc It is then natural to conjecture that visibility can always be localized, that is to say,
that a domain $\Omega$ is visible with respect to the Kobayashi-Royden metric of $\Omega$ 
and to $\partial \Omega$ (we will define this
precisely later) if and only if for any $p\in \partial \Omega$, 
there exists a neighborhood $U_p$ of $p$ such that $\Omega \cap U_p$ is visible 
with respect to the Kobayashi-Royden metric of $\Omega \cap U_p$ 
and to $U_p \cap \partial \Omega$. This problem was first studied in \cite{BGNT} and the authors were able to localize visibility under global assumptions such as Gromov hyperbolicity. \\
\spc The goal of this note is to show without  any global assumptions that visibility is indeed a local condition, depending on the boundary near a point. The main tool of the proof is additive localization of Kobayashi distance near visible points. These  localization results were recently established in \cite{S}. The paper is organized as follows. \\
\spc In   Section \ref{secdef}, we recall the definitions about Kobayashi distance and Kobayashi-Royden metric. Further, we introduce the notion of \emph{visible point} and
relate it to the visibility definitions in the literature. \\
\spc In Section \ref{secloc}, we prove Theorem \ref{localizationofkobayashidistance} which gives additive localization for Kobayashi distance near visible points. Our proof follows the proof given in \cite{S}, however only using local hypotheses. We also prove Theorem \ref{multiplicativelocalization}, which gives multiplicative localization of Kobayashi lengths near visible points. \\
\spc In section \ref{secvis}, we prove Theorem \ref{definitionsoflambdavisibilityareequivalent} that shows that local visibility at a boundary point is equivalent to global visibility at that point. In particular, it leads to Theorem \ref{visibilityisapropertyofthepoint} which tells us that if a boundary point is visible for a domain, then it is visible for any domain which locally looks like the initial one.
\section{Visible points and visibility property}
\label{secdef}
Let $\Omega$ be a domain in $\C^n$, $z,w\in \Omega$ and $v\in\Cn$.
Recall that the Kobayashi pseudodistance $k_\Omega$ is the largest pseudodistance
which does not exceed the Lempert function
$$ l_\O(z,w):=\tanh^{-1} \tilde{l}_\O(z,w),$$
where $\Delta$ is the
unit disc and $\tilde{l}_\Omega(z,w):=\inf\{|\alpha|:\exists\vp\in\hol(\Delta,\O)
\hbox{ with }\varphi(0)=z,\varphi(\alpha)=w\} $.\\  
\spc Also recall the definition of Kobayashi-Royden pseudometric,
$$\k_\O(z;v)=\inf\{|\alpha|:\exists\vp\in\O(\Delta,\O)\hbox{ with }
\varphi(0)=z,\alpha\varphi'(0)=v\}.$$
\spc Kobayashi-Royden length of an absolutely continuous curve $\gamma:I\rightarrow \O$ is defined as $$ \lk_{\O}(\gamma):=\int_{I} \k_\O(\gamma(t),\gamma'(t))dt. $$
\spc By \cite{R,V}, it turns out that $k_\Omega$ is the integrated form of the Kobayashi-Royden pseudometric. That is 
$ k_\O(z,w)=\inf\lk_{\O}(\gamma)$ where the infimum is taken over all absolutely continuous curve joining $z$ to $w$. \\
\spc We say that a domain $\O\subset \Cn$ is hyperbolic if $k_\O$ is a distance. This holds, for instance, for any bounded domain and for convex domains containing no affine complex lines. \\
\spc In order to study the local behaviour of invariant metrics,
we introduce a notion of hyperbolicity at boundary points. In fact, our definition is a generalization to boundary points of the characterization of hyperbolicity given in \cite[Proposition 3.1]{NP}.
\begin{defn}
	Let $\O$ be a domain in $\Cn$, $p\in\partial \O$. If $\O$ is bounded, we say that $\O$ is
	\emph{hyperbolic} at $p$ for any $p\in \partial \O$. If $\O$ is unbounded, we say that $\O$ is \emph{hyperbolic} at $p$
	if we have 
	\begin{equation}
		\label{newlocalhyperbolicity}
		\liminf_{z\to p,w\to \infty} l_\O(z,w) > 0. \end{equation}
\end{defn}
For $A,B\subset \O$ we denote $ l_\O(A,B):=\inf_{a\in A, b\in B} l_\O(a,b)$ and $k_\O(A,B):=\inf_{a\in A, b\in B} k_\O(a,b)$.
By definition \eqref{newlocalhyperbolicity} holds if and only if we can find a bounded neighbourhood $U$ of $p$ and another neighbourhood $V\subset\subset U$ of $p$ such that we have
\begin{equation}\label{localhyperbolicity1}
	l_\O(\O\cap V, \O\setminus U) > 0.
\end{equation}
\spc In fact, the latter condition holds on any bounded domain $\O$, for two arbitrarily small neighbourhoods $V\subset\subset U$ of $p\in\partial \O$. This is a motivation for why we set bounded domains to be hyperbolic at any boundary point. \\
\spc One can observe that this property is even stronger.

\begin{prop}\label{notionsoflocalhyperbolicity}
	Let $\O$ be an unbounded domain in $\Cn$ and $p\in\partial \O$. $\O$ is hyperbolic at $p$ if and only if there exists a bounded neighbourhood $V$ of $p$, such that for any  other two neighbourhoods of $p$ satisfying $V'\subset\subset U'\subset V$ we have 
	\begin{equation}\label{strongerlocalhyperbolicity}
		k_\O(\O\cap V', \O\setminus U') > 0.
	\end{equation}
\end{prop} 

\begin{proof}
	It is clear that if \eqref{strongerlocalhyperbolicity} holds, then we have \eqref{localhyperbolicity1} 
	with $V', U'$ playing the part of $V, U$, hence \eqref{newlocalhyperbolicity}. So we will prove the converse. \\
	Let $V\subset\subset U$ be two neighbourhoods of $p$ such that $U$ is bounded and \eqref{localhyperbolicity1} holds.
	Recall Royden's localization lemma\cite[Proposition 13.2.10]{JP}. 
	If $\O$ is a domain in $\Cn$ and $D$ is any subdomain, we have 
	\begin{equation}\label{roydenslocalizationlemma}
		\tilde{l}_\O(z,\O\setminus D) \k_{\O\cap D}(z;v)\leq \k_\O(z;v) \: \: \: \: \: \: z\in \O\cap V,\:\: v\in\Cn.
	\end{equation} 
	Note that by \eqref{localhyperbolicity1}, taking $D=U$ above, \eqref{roydenslocalizationlemma} implies that there exists a $C>0$ such that \begin{equation}\label{royden}
		\k_{\O\cap U}(z;v)\leq C \k_\O(z;v) \: \: \: \: \: \: z\in \O\cap V,\:\: v\in\Cn.
	\end{equation} 
	Let $V'\subset\subset U'\subset V$ be any two neighbourhoods of $p$. Since Kobayashi distance is the integrated version of Kobayashi-Royden metric,
	$k_\O(\O\cap V', \O \setminus U')=\inf_{\gamma\in \Gamma} \lk_{\O}(\gamma)$,
	where $\Gamma$ is the set of absolutely continuous curves $\gamma:I\rightarrow\O$ joining points in 
	$\O\cap V'$ to points in $\O\setminus U'$. 
	Due to connectivity of curves, it is clear that for any curve in $\Gamma$, we can find another curve in $\Gamma$ of shorther or equal length whose image is contained in $\ov V$. Therefore  
	$k_\O(\O\cap V',\O\setminus U')=\inf_{\gamma\in \Gamma'} \lk_{\O}(\gamma)$,
	where $\Gamma'\subset \Gamma$ is the set of such curves whose image lie in $\ov V$. 
	On the other hand, since $U$ is bounded, $\O\cap U$ is bounded, so 
	$$
	\inf_{\gamma\in\Gamma'} \lk_{\O\cap U}(\gamma) \geq k_{\O\cap U}(\O\cap V',\O\setminus U') = c > 0.
	$$ 
	\\
	By \eqref{royden} we obtain $$k_\O(\O\cap V',\O\setminus U')=\inf_{\gamma\in \Gamma} \lk_{\O}(\gamma)=\inf_{\gamma\in \Gamma'} \lk_{\O}(\gamma) \geq \dfrac{\inf_{\gamma\in\Gamma'} \lk_{\O\cap U}(\gamma)}{C}\geq\dfrac{c}{C} .$$ 
\end{proof} 
One may argue as above to provide a new proof of \cite[Proposition 3.1]{NP}.  \\
\spc Note that hyperbolicity does not imply local hyperbolicity without some assumptions. To see this, one may study the Kobayashi distance on the domain $D:=\{(z,w)\in \C^2: z\in \Delta\setminus\{0\}, |zw| < 1\}$ which is biholomorphic to 
$\Delta\setminus\{0\}\times \Delta$, hence pseudoconvex and hyperbolic. It is not difficult to see that $D$ is not hyperbolic at points $\{(0,z):z\in \C\}\subset\partial D$. 
\begin{defn} Let $\O$ be a domain in $\Cn$.
	For $\lambda \geq 1$ and $\epsilon \geq 0$ we say that an absolutely continuous curve $\gamma:I\rightarrow \O$ is a $(\lambda,\epsilon)$-geodesic for $\O$ if for all $t_1\leq t_2\in I$ we have that
	$$\lk_{\O}(\gamma |_{[t_1,t_2]}) \leq \lambda k_\O(\gamma(t_1),\gamma(t_2)) + \epsilon.$$ 
\end{defn}
We claim that in the case where $\O$ above is hyperbolic, this definition implies the definition of $(\lambda,\epsilon)$-almost geodesics given in \cite{BZ}. To see this notice that on hyperbolic domains any absolutely continuous curve can be reparametrized with respect to Kobayashi-Royden length, that is we can parametrize $\gamma$ so that $\lk_\O (\gamma|_{[t_1,t_2]}) = |t_1-t_2|$ for all $t_1\leq t_2 \in I$. In particular, we have $\k_\O(\gamma(t);\gamma'(t))=1$ almost everywhere so $\k_\O(\gamma(t);\gamma'(t))\leq \lambda$ almost everywhere. This can be shown with arguments similar to the arguments in the proof of \cite[Proposition 4.4.]{BZ}. Due to this parametrization and the definition of $(\lambda,\epsilon)$-geodesics we obtain 
$$  k_\O(\gamma(t_1),\gamma(t_2))\leq \lambda k_\O(\gamma(t_1),\gamma(t_2)) + \epsilon \leq \lambda \lk_\O(\gamma|_{[t_1,t_2]}) + \epsilon = \lambda|t_1-t_2|+\epsilon$$
and $$ k_\O(\gamma(t_1),\gamma(t_2)) \geq \lambda^{-1} \lk_\O(\gamma|_{[t_1,t_2]}) -\lambda^{-1}\epsilon = \lambda^{-1}|t_1-t_2| -\lambda^{-1}\epsilon \geq \lambda^{-1}|t_1-t_2|-\epsilon$$ so the claim follows. \\
\spc Unless it is otherwise noted, we will assume that any $(\lambda,\epsilon)$-geodesic on hyperbolic domains is parametrized as a $(\lambda,\epsilon)$-geodesic in the sense of \cite{BZ}. \\
\spc For brevity, we will say $\epsilon$-geodesics for $(1,\epsilon)$-geodesics. As the Kobayashi distance is given as the infimum of the Kobayashi-Royden lengths of the curves, we see that for any $\epsilon>0$ and any two points, we can find 
$\epsilon$-geodesics (hence, $(\lambda,\epsilon)$-geodesics) joining them. Also, using the triangle inequality, one may observe that for $\epsilon$-geodesics, if the condition in the definition is satisfied with the endpoints, then it holds everywhere. 
\\
\spc Let us introduce the notion of $\lambda$-visible points. 
Informally, being a $\lambda$-visible point means that a $(\lambda,\epsilon)$-geodesic 
with an extremity ``near" that point avoids the boundary immediately. Explicitly: 
\begin{defn} 
	Let $\O\subset \Cn$, $\lambda \geq 1$ and $\epsilon\geq 0$. We say that $p\in \partial \O$ is a $(\lambda,\epsilon)$-\emph{visible point for $\O$} if we have the following property: for any bounded neighbourhood $U$ of $p$, there exist $V\subset\subset U$ and a compactum $K_{\lambda,\epsilon}\subset\subset \O$, such that if $\gamma:I\rightarrow \O$ is a $(\lambda,\epsilon)$-geodesic for $\O$ which joins a point in $\O\cap V$ to a point in $\O\setminus U$, then $\gamma(I)\cap K_{\lambda,\epsilon}\neq\emptyset$. \\
	We say that $p\in\partial \O$ is a \emph{$\lambda$-visible point for $\O$} if $p$ is $(\lambda,\epsilon)$-visible  for $\O$ for all $\epsilon\geq 0$. \end{defn}
We also say that $p$ is a \emph{weakly visible point for $\O$} if it is $1$-visible for $\O$ and $p$ is a \emph{visible point for $\O$} if it is $\lambda$-visible for $\O$ for all $\lambda\geq 1$. \\
\spc It is clear that if $p$ is a $\lambda$-visible point for $\O$ and $\lambda'\leq \lambda$ then $p$ is a $\lambda'$-visible point for $\O$. \\
\spc Let $\O\subset\Cn$ be any domain and $z,w,o\in \O$. We recall the definition of Gromov product with respect to Kobayashi distance $$(z|w)^\O_o:=\frac{1}{2}(k_\O(z,o)+k_\O(w,o)-k_\O(z,w)) . $$
\begin{defn}
	Let $\O\subset \Cn$, $o\in \O$, $p\in\partial \O$. We say that $p$ satisfies the \emph{Gromov property for $\O$} if for any $q\neq p$ in $\partial \O$, we have that $$ \limsup_{z\to p,w\to q}(z|w)^\O_o < \infty .$$  
	We say that $p$ satisfies the \emph{weak Gromov property for $\O$} if for any $o\in \O$ there exists a constant $c \leq 0$ such that for any $q\neq p$ in $\ov \O$ we have $$\liminf_{z\to p,w\to q} \left( k_\O(z,w)-k_\O(z,o) \right) \geq c  .$$
\end{defn}
It is clear that the Gromov property does not depend on the choice of the base point $o$, and
that if $p$ satisfies the Gromov property for $\O$, then it satisfies the weak Gromov property for $\O$. \\
\spc Being a $\lambda$-visible point leads to the following conditions about growth of Kobayashi distance. 
\begin{prop}\label{visibilityandgrowthofkobayashi} Let $\O\subset \Cn$ be a hyperbolic domain and $p\in\partial \O$.
	\begin{enumerate}
		\item Suppose that $p$ is a $\lambda$-visible point for $\O$. Then $p$ satisfies the Gromov property for $\O$.  
		\item Suppose that $p$ is a $\lambda$-visible point for $\O$. Then $\O$ is hyperbolic at $p$.
	\end{enumerate}
\end{prop}
\begin{proof}
	The first statement follows from the proof of \cite[Proposition 2.4]{BNT}.\\
	To see that the second statement holds, note that by assumption for any bounded neighbourhood $U$ of $p$, we can find another neighbourhood of $p$, $V\subset\subset U$ and compact set $K\subset\subset\O$, where any $\epsilon$-geodesic joining points in $\O\cap V$ and $\O\setminus U$ meets $K$. It is well known that on hyperbolic domains the distance of any compact set to the boundary is bounded below. Therefore by shrinking $V$ if necessary and taking $\epsilon$ small enough, we obtain $$ 
	k_{\O}(\O\cap V,\O \setminus U)\geq k_{\O}(\O\cap V,K) -\epsilon > 0  .
	$$
	It follows by Proposition \ref{notionsoflocalhyperbolicity} that $\O$ is hyperbolic at $p$. 
\end{proof}
It is worth noting that by the proof of \cite[Proposition 2.5]{BNT}, if $\O$ is complete hyperbolic, that is if $(\O,k_\O)$ is complete as a metric space, then $p$ satisfies the Gromov property if and only if it is a weakly visible point. Moreover, statement (1) in Proposition \ref{visibilityandgrowthofkobayashi} hold even when the domain is not hyperbolic. 
\begin{defn}
	Let $\Omega, D$ be two hyperbolic domains in $\Cn$ which have a common boundary point $p$, and assume that both $\O$ and $D$ are hyperbolic at $p$. We say that \emph{$\Omega$ is equivalent to $D$ at $p$} if there exists a bounded neighbourhood $U\subset\Cn$ of $p$ such that $U\cap \Omega= U \cap D$.  
\end{defn}
It is easy to see the definition above indeed is an equivalence relation. Denote the equivalence class of $\Omega$ at $p$ by $[\Omega]_p$. In some sense, these equivalence classes can be considered as ``germs" of domains at points in $\Cn$. \\
\spc We recall the definition of visible pairs. Let $\O\subset \Cn$, $p\neq q \in \partial \O$. $\{p,q\}$ is said to be a \emph{visible pair} if there exists neighbourhoods $U_p$, $U_q$ of $p$ and $q$ respectively such that if $\gamma:I\rightarrow \O$ is a $(\lambda,\epsilon)$-geodesic joining a point in $U_p$ to point in $U_q$, then $\gamma(I)$ intersects a compact set $K_{\lambda,\epsilon}\subset\subset\O$. $\O$ enjoys the \emph{visibility property} if for any $p\neq q\in \partial \O$ we have that $\{p,q\}$ is a visible pair. \\
\spc With the same spirit, one can define $\lambda$-visible pairs and $\lambda$-visibility property if one restricts the definition above to $(\lambda',\epsilon)$-geodesics where $\lambda'\leq \lambda$. 
\begin{prop}\label{ourdefinitionandolderdefinition} Let $\O\subset \Cn$, $\lambda\geq 1$. A point $p\in\partial \O$ is a $\lambda$-visible  for $\O$ if and only if for any $q\in \partial{\O}$ such that $q\neq p$, $\{p,q\}$ is a  $\lambda$-visible  pair. Consequently, $\O$ enjoys $\lambda$-visibility property if and only if any $p\in \partial \O$ is a $\lambda$-visible point for $\O$.
\end{prop}
\begin{proof}
	It is clear that being visible at $p$ is a stronger assumption than any $\{p,q\}$ being visible pairs. So, we will prove the converse. \\
	We suppose that $p$ is not a visible point for $\O$. Then, there exists a bounded neighbourhood $U$ of $p$, and sequences $z_n\rightarrow p$, $w_n\in \O\setminus U$ and a sequence of $(\lambda,\epsilon)$-geodesics $\gamma_n:I_n\rightarrow \O$ joining $z_n$ to $w_n$ which eventually avoid any compact set in $\O$. By connectivity of geodesics, each $\gamma_n$ must meet $(\partial U)\cap \O$. Set $\tau_n:=\inf\{t\in I_n:\gamma_n(t)\in(\partial U)\cap \O\}$ and $w'_n=\gamma_n(\tau_n)$. By construction, $\gamma_n |_{[0,\tau_n]}$ are $(\lambda,\epsilon)$-geodesics for $\O$ joining $z_n$ to $w'_n$ which tend to $\partial \O$. Moreover, by our assumption passing to a subsequence if necessary we have $w'_n\rightarrow r\in\partial \O \cap \ov{U}$. This shows that $\{p,r\}\subset\partial\O$ is not a visible pair. 
\end{proof}

Proposition \ref{ourdefinitionandolderdefinition} shows, in some sense, that visibility for a domain is a local property of the boundary points. But the (almost) geodesics involved in the definition are constructed using the 
global metric; we will remove this restriction later, to give conditions depending only on the equivalence
class of $\O$ at each boundary point. 

\section{Localization of Kobayashi distance near visible points}
\label{secloc}
Let $\O$ be a domain in $\Cn$, $p\in\partial \O$. We say that $k_\O$ satisfies \emph{additive (resp. multiplicative) localization at $p$} if there exists neighbourhoods $V\subset\subset U$ of $p$, and $C\geq 0$ such that 
$$
k_{\O\cap U}(z,w)\leq k_\O (z,w)+ C\mbox{, resp. }
k_{\O\cap U}(z,w)\leq C k_\O (z,w)
$$
for any $z,w\in\O\cap V$. We would like to note that whenever we discuss such a property we will assume that $z,w\in \O\cap V$ belong to the same component of $\O\cap U$. \\
\spc This section is devoted to proving additive and multiplicative localization results near visible points. Although our results are given with weaker assumptions, they are of the same character as those  in \cite{S}. This is essentially due to Proposition \ref{ourdefinitionandolderdefinition}, which relates the notion of visible points to the notion of visible pairs. Even though our proofs closely follow \cite{S}, we discuss the arguments in detail to show the local nature, and also to simplify some steps of proofs given in \cite{S}. \\
\spc We are ready to introduce the following theorem, which is an analogue of \cite[Theorem 0.1, Theorem 0.2]{S}.
\begin{thm}\label{localizationofkobayashidistance}
	Let $\O\subset \Cn$ be a hyperbolic domain, $p\in\partial \Omega$ and $\lambda\geq 1$. Assume that one or both of the followings hold. 
	\begin{enumerate}
		\item $p$ is a $\lambda$-visible point for $\Omega$. 
		\item There exists a bounded neighbourhood of $p$, $U \subset \Cn$ such that $p$ is a $\lambda$-visible point for $\O\cap U$ and $\O$ is hyperbolic at $p$. 
	\end{enumerate}
	Then, there exists another neighbourhood $V\subset\subset U$ of $p$, and a constant $C\geq 0$ such that for any $z,w\in\O\cap V$ we have
	$$ k_{\O\cap U}(z,w)\leq k_\O(z,w)+C .$$
\end{thm}
In order to prove Theorem \ref{localizationofkobayashidistance} we will provide a localization result about Kobayashi-Royden lengths. Our result is an analogue of \cite[Lemma 0.8, Lemma 0.9]{S}.
\begin{lem}\label{localizationoflengths}
	Let $\O\subset \Cn$ be a hyperbolic domain and $p\in\partial \O$.
	\begin{enumerate}
		\item If $p$ is a $\lambda$-visible point for $\Omega$ for some $\lambda\geq 1$, then for any bounded neighbourhood $U$ of $p$ there exists $V\subset\subset U$ and a constant $C \geq 0$ such that for all $(\lambda,\epsilon)$-geodesics $\gamma:I\rightarrow\O\cap V$ for $\O$ we have $$\lk_{\Omega\cap U}(\gamma)\leq\lk_\Omega(\gamma)+C.$$
		\item If there exists a bounded neighbourhood of $p$, $U \subset \Cn$ and $\lambda\geq 1$ such that $p$ is a $\lambda$-visible point for $\O\cap U$ and $\O$ is hyperbolic at $p$, then there exists a neighbourhood $V\subset\subset U$ and a constant $C\geq 0$ such that for all $(\lambda,\epsilon)$-geodesics $\gamma:I\rightarrow\O\cap V$ for $\O$ we have $$\lk_{\Omega\cap U}(\gamma)\leq\lk_\Omega(\gamma)+C .$$
	\end{enumerate}
\end{lem} 
Proof of Lemma \ref{localizationoflengths} requires the reparametrization of $\lambda$-geodesics in the sense of almost-geodesics. As noted earlier, this is possible on hyperbolic domains. As the later results depend on Lemma \ref{localizationoflengths} we see that hyperbolicity assumption is essential. \\ 
We will need the following crucial estimate given in \cite{S}. It follows from Royden's localization lemma, by comparing the Kobayashi pseudodistance to the Lempert function and using the inequality $\tanh(x)\geq 1- e^{-x}$ for all $x\geq 0$.  
\begin{lem}\cite[Lemma 0.3]{S}
	Let $\O\subset \Cn$ be a hyperbolic domain and $V\subset\subset U$ be neighbourhoods of $p\in\partial \Omega$ such that $$ k_\O(\O\cap V,\O\setminus U) > 0.$$ 
	Then there exists $C>0$ such that the infinitesimal Kobayashi metric satisfies  
	\begin{equation}\label{sarkarsestimate}
		\k_{\Omega\cap U}(x;v)\leq (1+C e^{-k_\O(x,\Omega\setminus U)})\k_\Omega(x;v)
	\end{equation} 
	for all $x\in \O\cap V$, $v\in \Cn$. 
\end{lem}
It follows by \cite{S} that $C$ above can be taken to be $\coth(k_\O(\O\cap V,\O\setminus U))$. \\  
\spc \textit{Proof of Lemma \ref{localizationoflengths}.}
Let $U$ be as in Lemma \ref{localizationoflengths}. Proposition \ref{notionsoflocalhyperbolicity} and our assumption in (2) of Lemma \ref{localizationoflengths} imply that on both cases, there exists a neighbourhood $V\subset\subset U$ of $p$ such that $$ k_\O(\O\cap V,\O\setminus U) > 0 .$$ Let $V$ be such a set, assume that $z,w\in \Omega\cap V$ and $\gamma: I \rightarrow \Omega\cap V$ is a $(\lambda,\epsilon)$-geodesic for $\O$ joining them.
By a direct calculation, using \eqref{sarkarsestimate} we see that 
$$
\lk_{\Omega\cap U}(\gamma)=\int_{I} \k_{\Omega\cap U}(\gamma(t);\gamma'(t))dt \leq 
\int_{I} (1+C e^{-k_\O(\gamma(t),\Omega\setminus U)})\k_\Omega(\gamma(t);\gamma'(t))dt $$
\begin{equation}\label{estimatingthelength}
	\leq \lk_\Omega(\gamma)+C\int_{I}  e^{-k_\O(\gamma(t),\Omega\setminus U)}\k_\Omega(\gamma(t);\gamma'(t))dt
\end{equation} 
Set $E:= \int_{I}  e^{-k_\O(\gamma(t),\Omega\setminus U)}\k_\Omega(\gamma(t);\gamma'(t))dt$. We will show that $E$ is uniformly bounded above for any such curve. \\
\emph{We first assume that $p$ is a $\lambda$-visible point for $\O$.} \\ 
By Proposition \ref{visibilityandgrowthofkobayashi}, $p$ satisfies the weak Gromov property for $\O$. 
Notice that this property implies that for any bounded neighbourhood $U$ of $p$ we can find another neighbourhood $V\subset\subset U$ such that for a given $o\in\O$, we have a constant $c\leq 0$ such that 
$$k_\O(z,w)\geq k_\O(z,o)+c$$ for all $z\in \O\cap V$, $w\in \O\setminus U$. \\ In particular, shrinking $V$ if necessary and fixing $o\in \O\cap V$ we can find a constant $c\leq 0$ satisfying
$$k_\O(\gamma(t),\Omega\setminus U):=\inf_{w\in \Omega\setminus U} k_\O (\gamma(t),w) \geq k_\O (\gamma(t),o)+c $$  for all $t\in I$ .
Thus, we obtain \begin{equation}\label{boundingtheerror}E\leq C'\int_{I}  e^{-k_\O(\gamma(t),o)}\k_\Omega(\gamma(t);\gamma'(t))dt.
\end{equation}
We reparametrize $\gamma$ so that it becomes an $(1,\epsilon)$-almost geodesic
for $\k_\O$, i.e. $\k_\O(\gamma(t);\gamma'(t))=1$ almost everywhere. 
Choose $t_0\in I$ with $k_\O(\gamma(t_0),o))\leq k_\O(\gamma(t),o))$ for all $t\in I$. Thus
$$ k_\O(\gamma(t),o)\geq \dfrac{1}{2} \left( k_\O(\gamma(t_0,o)+k_\O (t,o)) \right) \geq \dfrac{1}{2} k_\O (\gamma(t_0),\gamma(t)) \geq \dfrac{1}{2}\left(\frac{|t-t_0|}{\lambda}-\epsilon\right).$$

Using this fact again, by \eqref{boundingtheerror} we obtain 
$$ 
E\leq   C''\int_{I} e^{-\frac{|t-t_0|}{2\lambda}} \k_\Omega(\gamma(t);\gamma'(t))dt \leq 2 C'' \int_{\mathbb{R}} e^{-\frac{t}{2\lambda}}dt \leq C''' $$ so by \eqref{estimatingthelength} we have $$\lk_{\Omega\cap U}(\gamma)\leq \lk_\O (\gamma)+C'''.
$$ 
This finishes the case where $p$ is $\lambda$-visible for $\O$. 
\\
\emph{Now we assume $p$ is $\lambda$-visible for $\O\cap U$ and $\O$ is hyperbolic at $p$.} \\ By the hyperbolicity assumption at $p$, choosing $V$ as before we deduce that \eqref{estimatingthelength} still holds and we want to show that $E:= \int_{I} e^{-k_\O(\gamma(t),\Omega\setminus U)}\k_\Omega(\gamma(t);\gamma'(t))dt$ is uniformly bounded. \\
To do so, we take a $W\subset\subset U$ satisfying $V\subset\subset W$ and clearly we have that 
\begin{equation} \label{obviously}
	k_\O(\gamma(t),\Omega\setminus U)\geq k_\O (\gamma(t),\Omega\setminus W).
\end{equation} 
{\it Claim.} There are constants $c_1,c_2$ so that $k_\O (\gamma(t),\Omega\setminus W)\geq  c_1 k_{\O\cap U}(\gamma(t),\O\setminus W)+ c_2$. 

This is proven in \cite{S} but for convenience we repeat the arguments. \\
Note that as $\O$ is hyperbolic at $p$, by Proposition \ref{notionsoflocalhyperbolicity} shrinking $V,W$ if necessary we can find a constant $C\geq 1$ so that \begin{equation}\label{comparingtheinfinitesimal}
	\k_{\O\cap U}(z;v)\leq C \k_{\O}(z;v)	\:\: \: \: \: \: z\in \O \cap \overline W, \: v\in \Cn
\end{equation}  
Fix $\delta,\epsilon >0$. For each $t$, choose $w_t\in \O\setminus W$ such that 
$ k_\O (\gamma(t),w_t )\leq k_\O(\gamma(t),\O\setminus W)+\delta $ 
and let $\sigma_t:I_t\rightarrow\O$ be an $\epsilon$-geodesic for $\O$ joining $\gamma(t)$ to $w_t$. 
Set $\tau_t:=\inf\{t'\in I_t: \sigma_t(t')\in\O\setminus W\}$ 
and $w'_t:=\sigma_t(\tau_t)$ and set $\sigma'_t:=\sigma_t |_{[0,\tau_t]}$. 
\\
Then by \eqref{comparingtheinfinitesimal} we obtain 
$$ k_{\Omega\cap U}(\gamma(t),\O\setminus W)\leq k_{\O\cap U}(\gamma(t),w'_t)\leq\lk_{\Omega\cap U}(\sigma'_t)\leq C \lk_\Omega(\sigma'_t) $$ $$\leq C\lk_\Omega(\sigma_t) \leq Ck_\O(\gamma(t),\O\setminus W)+C\epsilon+C\delta .$$ The estimate above finishes the proof of the claim. \\ Now, by \eqref{obviously} and the claim we obtain $$E \leq C' \int_{I} e^{-C''k_{\O\cap U}(\gamma(t),\O\setminus W)}\k_\O(\gamma(t);\gamma'(t))dt .$$
To continue, one can repeat the proof of the case where $p$ is $\lambda$-visible for $\O$ by applying weak Gromov property of $p$ with respect to $\O\cap U$ instead of $\O$ to see that also in this case $E$ is bounded. \\ By \eqref{estimatingthelength}, the theorem follows.
$\endofproof$ \\
\spc Lemma \ref{localizationoflengths} and its proof immediately lead to the following useful corollary. 
\begin{cor}\label{biggeodesicissmallgeodesic}
	Let $\O$ be a hyperbolic domain in $\Cn$, $p\in\partial \O$, and let $V\subset\subset U$ be two neighbourhoods of $p$ such that $U$ is bounded and we have $$ k_\O(\O\cap V,\O\setminus U) > 0.$$ Let $z,w\in \O\cap V$ and assume that $\gamma$ is an $(\lambda,\epsilon)$-geodesic for $\O$ joining $z$ to $w$ whose image lies in $\O\cap V$. Then there exists an $\epsilon'\geq \epsilon$ such that $\gamma$ is an $(\lambda,\epsilon')$-geodesic for $\O\cap U$. 
	
\end{cor}
\begin{proof}
	Observe that in the proof of Lemma \ref{localizationoflengths} we showed that if $\gamma$ is an $(\lambda,\epsilon)$-geodesic for $\O$ that lies in $\O\cap V$ then there exists a constant $C \geq 0$ such that $\lk_{\Omega\cap U}(\gamma)\leq \lk_\O(\gamma)+C $ and $C$ depends on $U,V$ and $\lambda$. \\
	Now, as $\gamma$ is an $(\lambda,\epsilon)$-geodesic for $\O$ we have 
	$$ \lk_{\O\cap U}(\gamma)\leq \lk_\O(\gamma) + C \leq \lambda k_\O(z,w) + C + \epsilon \leq \lambda k_{\O\cap U}(z,w)+ C + \epsilon .$$ Setting $\epsilon'=\epsilon+ C$ we are done. 
\end{proof}
\textit{Proof of Theorem \ref{localizationofkobayashidistance}.}
\emph{We first assume that $p$ is a $\lambda$-visible point for $\O$.} \\
Let $U$ be a bounded neighbourhood of $p$. By Proposition \ref{visibilityandgrowthofkobayashi}, we have that $\O$ is hyperbolic at $p$ so we can find a $V\subset\subset 
U$ satisfying $$k_\O(\O\cap V,\O\setminus U) > 0 .$$\\
We first assume that $\gamma$ is an $\epsilon$-geodesic for $\O$ which joins $z,w\in \O\cap V$. 
Observe that if the image of $\gamma$ lies entirely in $\O\cap V$ by Lemma \ref{localizationoflengths}  we have 
$$ k_{\O\cap U}(z,w) \leq \lk_{\O\cap U}(\gamma) \leq \lk_\O(\gamma)+C \leq k_{\O}(z,w)+\epsilon+C $$ so in this case the localization holds. \\
Now to get a contradiction, we assume that there exists $z_n,w_n\subset \O\cap V$ such that 
$$ k_{\O\cap U}(z_n,w_n)-k_{\O}(z_n,w_n)\rightarrow \infty .$$ 
Let $\gamma_n:I_n\rightarrow \O$ be $\epsilon$-geodesics for $\O$ joining $z_n$ to $w_n$. By the observation above, by looking at a subsequence if necessary we observe that we must have $\gamma_n(I_n)\not\subset V$. Set $\tau^1_n:=\inf\{t\in I_n: \gamma_n(t)\in \O\setminus V\}$, $z'_n:=\gamma_n(\tau^1_n)$ and $\tau^2_n:=\sup\{t\in I_n:= \gamma_n(t)\in\O\setminus V\}$, $w'_n:=\gamma_n(\tau^2_n)$. Let $\gamma^1_n$,$\gamma^2_n$ be the parts of $\gamma_n$ that join $z_n$ to $z'_n$ and $w_n$ to $w'_n$ respectively.\\
As the curves $\gamma_n$ are $\epsilon$-geodesics for $\O$, by the above we have 
$$
k_\O(z_n,w_n)+\epsilon \geq \lk_\O(\gamma) = \lk_\O(\gamma^1_n)+\lk_\O(\gamma^2_n)+ \lk_\O(\gamma_n\setminus{\gamma^1_n\cup\gamma^2_n}) $$ 
$$\geq \lk_{\O\cap U}(\gamma^1_n)+\lk_{\O\cap U}(\gamma^2_n)+k_\O(z'_n,w'_n)-2C $$ \begin{equation}\label{laststepoflocalization}
	\geq k_{\O \cap U}(z_n,z'_n) + k_{\O \cap U}(z'_n,w'_n)+k_\O(z'_n,w'_n)-2C
\end{equation}
Suppose that $\{z'_n,w'_n\}_{n\in \mathbb{N}}$ is compact in $\O\cap U$. Then adding and substracting $k_{\O\cap U}(z'_n,w'_n)$ to the above inequality, by the triangle inequality, our assumption fails. Therefore, by looking at a subsequence if necessary, we  assume $z'_n\rightarrow s, w'_n\rightarrow r$ and either $s$, $r$ or both lie in $(\partial \O)\cap U$. As the approach is the same for all cases, we will provide the proof for the case where both $s,r\in (\partial \O)\cap U$. \\
Since $p$ is a visible point for $\O$, we can find $U'\subset\subset V$ such that any $\epsilon$-geodesic for $\O$ joining points in $\O\cap U'$ to points in $\O\setminus V$ meets a compactum $K\subset\subset\O$. By looking at a subsequence if necessary, we assume that $z_n,w_n\in U'$, so weak visibility of $p$ for $\O$ applies to the $\epsilon$-geodesics $\{\gamma^1_n,\gamma^2_n\}_{n\in\mathbb{N}}$. Thus, we can take $\tilde{z}_n,\tilde{w}_n$ so that they remain in the intersection of images of $\gamma^1_n,\gamma^2_n$ with $K$. By construction images of $\gamma^1_n,\gamma^2_n$ remain in $\O\cap V$ so we have $\{\tilde{z}_n,\tilde{w}_n\}_{n\in \mathbb{N}}\subset\subset \O\cap U$.  
\\
A similar calculation to \eqref{laststepoflocalization} leads to 
$$	k_\O(z_n,w_n)+\epsilon \geq 
k_{\O \cap U}(z_n,\tilde{z}_n) + k_{\O \cap U}(w_n,\tilde{w}_n)+k_\O(\tilde{z}_n,\tilde{w}_n)-2C
$$ 
$$
\geq k_{\O \cap U}(z_n,\tilde{z}_n) + k_{\O \cap U}(w_n,\tilde{w}_n)+k_\O(\tilde{z}_n,\tilde{w}_n) 
+ k_{\O\cap U}(\tilde{z}_n, \tilde{w}_n)-k_{\O \cap U}(\tilde{z}_n,\tilde{w}_n) -2C 
$$ 
$$
\geq k_{\O\cap U}(z_n,w_n)-C' ,
$$
where the last inequality follows from the triangle inequality and the fact that 
$\{\tilde{z}_n,\tilde{w}_n\}_{n\in \mathbb{N}}\subset\subset \O\cap U$. 
We therefore conclude that in the case where $p$ is $\lambda$-visible for $\O$, such a sequence can not exist. 
\\
\emph{We now assume that $p$ is $\lambda$-visible for $\O\cap U$ and $\O$ is hyperbolic at $p$.} \\ We choose $V\subset\subset U$ such that $k_\O(\O\cap V,\O\setminus U)>0$. To get a contradiction, we assume that localization fails. Thus, we can find points $z_n,w_n$ in $\O\cap V$ such that $$k_{\O\cap U}(z_n,w_n)-k_\O(z_n,w_n)\rightarrow \infty .$$ Repeating the construction above we see that $\epsilon$-geodesics joining $z_n$ to $w_n$ must leave $\O\cap V$ and \eqref{laststepoflocalization} holds also on this case. We may apply a similar visibility argument however we claim that using the weak Gromov property of $p$ with respect to $\O\cap U$ is enough. To see this observe that by \eqref{laststepoflocalization} we have 
\begin{equation}\label{lastlaststepreally}
	k_\O(z_n,w_n)\geq k_{\O\cap U}(z_n,z'_n)+ k_{\O\cap U}(w_n,w'_n) -2C .\end{equation}
Fix $o\in \O\cap V$ and shrink $V$ if necessary to get by weak Gromov property of $p$ with respect to $\O\cap U$ a constant $c\leq 0$ such that $$k_{\O\cap U}(z_n,z'_n)\geq k_{\O\cap U}(z_n,o) + c \: \text{and} \: k_{\O\cap U}(w_n,w'_n) \geq k_{\O\cap U}(w_n,o) + c .$$ Having this in mind and continuing \eqref{lastlaststepreally}, by triangle inequality we obtain $$k_\O(z_n,w_n)\geq k_{\O\cap U}(z_n,z'_n)+ k_{\O\cap U}(w_n,w'_n) -2C $$ $$ \geq k_{\O\cap U}(z_n,o)+k_{\O\cap U}(w_n,o)- 2C+2c \geq k_{\O\cap U}(z_n,w_n)-2C+2c .$$ 
We see that such a sequence cannot exist. This finishes the proof of this case. Hence, we have the theorem. 
$\endofproof$ \\
\spc Theorem \ref{localizationofkobayashidistance} gives the following corollary, which can be seen as a converse to Corollary \ref{biggeodesicissmallgeodesic}.
\begin{cor}\label{smallgeodesicisbiggeodesic}
	Let $\O$ be a hyperbolic domain and $p\in\partial \O$. Let $V\subset\subset U$ be two neighbourhoods of $p\in\partial \Omega$ such that additive localization holds for $U,V$, that is, there exists a $C\geq 0$ such that 
	$$ k_{\O\cap U}(z,w)\leq k_{\O}(z,w)+ C $$ for all $z,w\in \O\cap V$. For any $\epsilon > 0$, there exists an $\epsilon' \geq \epsilon$ such that any $(\lambda,\epsilon)$-geodesic $\gamma:I\rightarrow \O\cap V$ for $\Omega\cap U$ is a $(\lambda,\epsilon')$-geodesic for $\O$. 
\end{cor}
\begin{proof}
	Let $\gamma$ be as above and let $z,w$ be endpoints of the image of $\gamma$. Then by Theorem \ref{localizationofkobayashidistance} we obtain
	$$ \lk_{\O}(\gamma)\leq \lk_{\O\cap U}(\gamma) \leq \lambda \k_{\O\cap U}(z,w)+\epsilon \leq \lambda \k_\O(z,w) + \lambda C+\epsilon.$$
	Setting $\epsilon':=\epsilon+ \lambda C$, we are done. 
\end{proof}
One may observe that the weak Gromov property plays a key role in the proofs above. In fact, statements (2) of both Theorem \ref{localizationofkobayashidistance}  and Lemma \ref{localizationoflengths} still hold when one replaces "being a visible point for" with "satisfying weak Gromov property for". Moreover, statement (1) of Lemma \ref{localizationoflengths} also holds under the same change, however with the additional assumption that $\O$ being hyperbolic at $p$. It is unclear that if conditions in statement (1) of Theorem \ref{localizationofkobayashidistance} can be relaxed. \\
\spc As a conclusion for this section we will present a new multiplicative localization result. 
\begin{thm}\label{multiplicativelocalization}
	Let $\O\subset \Cn$ be a hyperbolic domain, $p\in\partial \Omega$. Assume that one or both of the followings hold. 
	\begin{enumerate}
		\item $p$ is a $\lambda$-visible point for $\Omega$. 
		\item There exists a neighbourhood of $p$, $U \subset \Cn$ such that $p$ is a $\lambda$-visible point for $\O\cap U$ and $\O$ is hyperbolic at $p$. 
	\end{enumerate}
	Then there exists a neighbourhood $V\subset\subset U$ of $p$ and a constant $C\geq 1$ such that 
	$$ k_{\O\cap U}(z,w)\leq C k_\O (z,w) $$ for all $z,w\in\O\cap V$.  
\end{thm}
\begin{proof}
	Due to Proposition \ref{visibilityandgrowthofkobayashi} and the assumption in (2) on both cases we can find a neighbourhood $V\subset\subset U$ of $p$ such that $$k_\O(\O\cap V,\O\setminus U)>0.$$ Let $V$ be such a set. It follows from the proof of Lemma \ref{localizationoflengths} that, if there are $\epsilon_n$-geodesics $\gamma_n$ joining $z\in \O\cap V$ to $w\in \O\cap V$ such that $\epsilon_n\rightarrow 0$ and the images of $\gamma_n$ remain in $\O\cap V$ there is a constant $C \geq 1$ such that 
	$$ k_{\O\cap U}(z,w)\leq C k_{\O}(z,w) .$$ so in this case multiplicative localization holds.  \\
	Let $V'\subset\subset V$ be a neighbourhood of $p$ such that there exists a constant $C_{V'}\geq 0$ such that we have $$ k_{\O\cap U}(z,w)\leq k_\O(z,w)+ C_{V'} $$ for any $z,w\in\O\cap V'$. Further, assume that $k_\O(\O\cap V',\O\setminus V)=:c\geq 0$. Note that we can choose such a neighbourhood due to Proposition \ref{notionsoflocalhyperbolicity} and Theorem \ref{localizationofkobayashidistance}. \\ 
	We suppose that $\gamma:I\rightarrow\O$ is an $\epsilon$-geodesic for $\O$ joining $z\in \O\cap V'$ to $w\in\O\cap V'$ where $\epsilon < c$ which leaves $\O\cap V$. Set $\tau:=\inf\{t\in \O: \gamma(t)\in \O\cap\partial V\}$ and $w':=\gamma(\tau)$. 
	Then, $$c < k_\O(\O\cap V', \O\setminus V)\leq k_\O(z,w') \leq \lk_{\O}(\gamma|_{[0,\tau]})\leq \lk_{\O}(\gamma) \leq k_\O(z,w) + \epsilon.$$ By above we obtain that $
	k_\O(z,w)\geq c-\epsilon$.
	As additive localization holds on $V'$, we have $$\dfrac{k_{\O\cap U}(z,w)}{k_\O(z,w)}\leq 1+ \dfrac{C_{V'}}{k_\O(z,w)}\leq 1+\dfrac{C_{V'}}{c-\epsilon} $$
	so $k_{\O\cap U}(z,w)\leq C' k_\O(z,w)$. Setting $V:=V'$ in the theorem, we are done. 
\end{proof}
By looking at the proof above, one may observe that if $k_\O(z,w)$ is bounded below, then additive localization is stronger than multiplicative localization. On the other hand, it is clear that if $k_\O(z,w)$ tends to $0$, multiplicative localization is a much better tool to compare $k_\O$ with $k_{\O\cap U}$. 
\section{Local visibility and global visibility}
\label{secvis}
The goal of this section is to show that visibility is a local property. Unlike the results given in \cite{BGNT}, we do not have any global assumptions such as Gromov hyperbolicity. Notably, the key elements in our proofs are Corollary \ref{biggeodesicissmallgeodesic} and Corollary \ref{smallgeodesicisbiggeodesic}. We will first prove the following:
\begin{thm}\label{definitionsoflambdavisibilityareequivalent}
	Let $\O\subset \Cn$ be a hyperbolic domain, $\lambda\geq 1$ and assume that $\O$ is hyperbolic at $p\in\partial \O$. Then the followings are equivalent. 
	\begin{enumerate}
		\item $p$ is a $\lambda$-visible point for $\O$.
		\item There exists a bounded neighbourhood  $U$ of $p$ such that $p$ is a $\lambda$-visible point for $\O\cap U$. 
		\item For any bounded neighbourhood $U$ of $p$, $p$ is a $\lambda$-visible point for $\O\cap U$. 
	\end{enumerate}
\end{thm}
\begin{proof}
	The implication that (3) $\implies$ (2) is clear. \\
	We will show that (2) implies (1) by contradiction. To see this assume that $p$ is not $\lambda$-visible for $\O$. Then by Proposition \ref{ourdefinitionandolderdefinition} we can find a point $q\in\partial \O$ such that $q\neq p$ and sequences $z_n\rightarrow p$, $w_n\rightarrow q$ and $(\lambda,\epsilon)$-geodesics $\gamma_n:I_n\rightarrow \O$ joining $z_n$ to $w_n$ such that $\gamma_n(I_n)$ avoids $K_n$ where the sequence $\{K_n\}_{n\in\mathbb{N}}$ satisfy $K_n\subset\subset K_{n+1} \subset\subset \O$ and $\O = \bigcup_n K_n$. 
	Now, take another neighbourhood $V\subset\subset U$ of $p$. By choosing $V$ small enough we assume that $w_n\notin V$ and by hyperbolicity of $\O$ at $p$, we can also assume that $k_\O(\O\cap V,\O\setminus U) > 0$. Set $\tau_n:=\inf\{t\in I_n: \gamma_n(t)\in (\partial V)\cap \O\}$ and $w'_n:=\gamma_n(\tau_n)$. Due to our assumption, passing to a subsequence if necessary, we assume $\lim_{n\rightarrow\infty} w'_n:=r \in (\partial\O)\cap \ov{V}$. \\
	Now, consider the curves $\gamma_n |_{[0,\tau_n]}$. By 
	construction they are $(1,\epsilon)$-geodesics for $\O$ whose image remain in $\O\cap V$. By Corollary \ref{biggeodesicissmallgeodesic}, there is an $\epsilon' > 0$ such they are $(1,\epsilon')$-geodesics for $\O\cap U$. We observe that for the pair $\{p,r\}\subset(\partial \O\cap U)$, we can find $(\lambda,\epsilon')$-geodesics for $\O\cap U$, joining points tending to $p$ and $r$ respectively. By construction, they eventually avoid any compact set in $\O\cap U$. This shows that $\{p,r\}$ is not a $\lambda$-visible pair for $\O\cap U$. By Proposition \ref{ourdefinitionandolderdefinition}, $p$ is not a $\lambda$-visible point for $\O\cap U$. This is a contradiction, we must have (2) $\implies$ (1). \\
	Now we will show that (1) $\implies$ (3). The proof is very similar to above. \\
	Let $U$ be any bounded neighbourhood of $p$. To get a contradiction we suppose that there exists sequences $z_n\rightarrow p$, $w_n\rightarrow q\in \partial{\O\cap U}$ with $q\neq p$, and $(\lambda,\epsilon)$-geodesics for $\O\cap U$ $\gamma_n:I_n\rightarrow \O\cap U$ joining $z_n$ to $w_n$ eventually avoiding any compact set in $\O\cap U$. \\
	By Theorem \ref{localizationofkobayashidistance}, we can find another neighbourhood $V\subset\subset U$ of $p$ and a constant $C\geq 0$ satisfying $$k_{\O\cap U}(z,w)\leq k_\O(z,w)+C $$ for any $z,w\in \O\cap V$. \\
	By shrinking $V$ if necessary, we assume that $w_n\notin \O\cap V$. As above, we consider the restrictions of image of $\gamma_n$ to $\O\cap V$. Denote those curves $\gamma'_n$ and set $w'_n$ to be the point where $\gamma'_n$ leaves $V$. Our assumption gives that by passing to a subsequence if necessary we have $w'_n\rightarrow r \in \overline{V}\cap \partial \O$. By Corollary \ref{smallgeodesicisbiggeodesic}, there exists an $\epsilon'$ such that each $\gamma'_n$ is a $(\lambda,\epsilon')$-geodesic for $\O$.  We observe that we can find $(\lambda,\epsilon')$-geodesics for $\O$ joining $\{p,r\}\subset \partial \O$ which tend to $\partial\O$. Thus, $\{p,q\}$ is not a $\lambda$-visible pair for $\O$. By Proposition \ref{ourdefinitionandolderdefinition}, this contradicts with the fact that $p$ is a $\lambda$-visible point for $\O$, and we see that (1) $\implies$ (3). We are done. 
\end{proof}
By Theorem \ref{definitionsoflambdavisibilityareequivalent} it is easy to see that $\lambda$-visibility of a boundary point is independent from the domain itself. More formally we have:
\begin{thm}\label{visibilityisapropertyofthepoint} 	Let $\O\subset \Cn$ be a hyperbolic domain which is hyperbolic at $p\in\partial \O$. Then $p$ is a  $\lambda$-visible point for $\O$ if and only if for any $D\in [\O]_p$, we have that $p$ is a $\lambda$-visible point for $D$.   
\end{thm}
\begin{proof}
	The second statement clearly implies the first. \\
	To see the converse, we assume that $p$ is visible for $\O$. Let $D \in [\O]_p$. Then there exists a bounded neighbourhood $U$ of $p$ such that $\O\cap U = D\cap U$. Theorem \ref{definitionsoflambdavisibilityareequivalent} implies that $p$ is a visible point for $D\cap U$. By definition $D$ is also hyperbolic at $p$ so again by Theorem \ref{definitionsoflambdavisibilityareequivalent} we see that $p$ is a visible point for $D$. We are done.  
\end{proof}


\end{document}